\documentclass[11pt]{article}
\usepackage[a4paper,margin=1in]{geometry}
\usepackage{amsmath,amssymb,amsthm,mathtools}
\usepackage{enumitem}
\usepackage{microtype}
\usepackage[hidelinks]{hyperref}
\usepackage[nameinlink,noabbrev]{cleveref}

\newtheorem{theorem}{Theorem}[section]
\newtheorem{proposition}[theorem]{Proposition}
\newtheorem{lemma}[theorem]{Lemma}
\newtheorem{corollary}[theorem]{Corollary}
\newtheorem{problem}[theorem]{Problem}
\newtheorem{ASdescent}[theorem]{Artin--Schreier descent lemma}
\theoremstyle{remark}
\newtheorem{remark}[theorem]{Remark}

\newcommand{\Aut}{\operatorname{Aut}}
\newcommand{\Res}{\operatorname{Res}}
\newcommand{\Pic}{\operatorname{Pic}}
\newcommand{\Dih}{\operatorname{Dih}}
\newcommand{\F}{\mathbb F}
\newcommand{\doi}[1]{\href{https://doi.org/#1}{doi:#1}}

\title{Nakajima-extremal Artin--Schreier covers of ordinary elliptic curves in characteristic $2$}
\author{Saeed Tafazolian}
\date{}

\begin{document}
\maketitle

\begin{abstract}
Let $k$ be an algebraically closed field of characteristic $2$ and let
$q=2^h$, $h\ge3$.  We construct ordinary bielliptic curves $X$ of genus
$q+1$ for which
\[
        \Aut(X)\cong \Dih(C_q)\times C_2,
        \qquad |\Aut(X)|=4q=4(g(X)-1).
\]
These curves realize case \textup{(ib)} in the classification of
Giulietti--Korchm\'aros and give an infinite family answering a problem
posed by Korchm\'aros.  More generally, we prove that every curve in case
\textup{(ib)} arises from the same construction.  Case \textup{(ib)}
occurs exactly in genera $g=2^h+1$ with $h\ge2$.  For $g\ge9$ we determine
the full automorphism group; in genus $5$ we determine its Sylow
$2$-subgroup but do not claim the full automorphism group.  For every fixed
$q\ge8$, the isomorphism classes in case
\textup{(ib)} of genus $q+1$ are parametrized bijectively by
$(k^\times)^2$.  The construction is described in terms of an ordinary
elliptic curve and an invariant differential.  We determine
the short orbits and ramification, the unramified cyclic quotients, and
the quotients by the central involutions.  For $q=8$ an explicit plane
model over $\F_2$ is given.
\end{abstract}

\section{Introduction}

Throughout the paper, $k$ is an algebraically closed field of
characteristic $2$.  Let $X$ be a curve of genus $g\ge2$, let $\gamma$
be its $2$-rank, and let $S$ be a $2$-subgroup of $\Aut(X)$.  We consider
the hypotheses
\[
\begin{array}{ll}
\textup{(I)}  & |S|\ge8,\qquad |S|>2(g-1),\\[2mm]
\textup{(II)} & S\text{ fixes no point of }X.
\end{array}
\]
Giulietti and Korchm\'aros determined the possibilities for $(X,S)$
under these assumptions \cite{GK2015}.  A central involution $v\in Z(S)$
is called \emph{inductive} if $X/\langle v\rangle$ has genus at least $2$
and the induced group $S/\langle v\rangle$ again satisfies
\textup{(I)} and \textup{(II)}.  We recall their classification theorem.

\begin{theorem}[Giulietti--Korchm\'aros {\cite[Theorem~1.2]{GK2015}}]
\label{thm:GKclassification}
Assume that $S$ satisfies \textup{(I)} and \textup{(II)}.  Then one of
the following cases occurs.
\begin{enumerate}[label=\textup{(\roman*)}]
\item
\[
        |S|=4(g-1),\qquad \gamma=g,
\]
and $X$ is bielliptic.  Moreover, either
\begin{enumerate}[label=\textup{(i\alph*)}]
\item $S$ is dihedral and has no inductive central involution; or
\item
\[
        S=(C_0\times\langle u\rangle)\rtimes\langle w\rangle,
\]
where $C_0$ is cyclic of order $g-1$, $u$ is a central involution,
$w$ is an involution, and $S/\langle u\rangle$ is dihedral.  Moreover,
$S$ has exactly two central inductive involutions, both contained in
$C_0\times\langle u\rangle$.
\end{enumerate}
\item One has $\gamma=g$ and
\[
        |S|=2g+2,
\]
and $S=D\rtimes\langle w\rangle$, where $D$ is an elementary abelian
subgroup of index $2$ and $w$ is an involution.  If $S$ is abelian,
then it is elementary abelian and $X$ is hyperelliptic.
\item Every central involution of $S$ is inductive.
\end{enumerate}
\end{theorem}

\begin{proposition}[Refinement of case \textup{(ib)} in \cite{GK2015}]
\label{prop:GK-refinement}
In case \textup{(ib)}, let $a$ be the unique involution of $C_0$.
Then $a$ and $au$ are the two inductive central involutions of $S$, while
$u$ is non-inductive; in fact
\[
        g(X/\langle u\rangle)=1.
\]
\end{proposition}

\begin{proof}
This is the refinement established in the proof of
\cite[Theorem~4.3, pp.~277--278]{GK2015}: the involutions $a$ and $au$
are shown there to be inductive, whereas the quotient by $u$ has genus
$1$.
\end{proof}

\begin{remark}\label{rem:GK-reading}
The published statement of \cite[Theorem~1.2(ib)]{GK2015} says that
``the two involutions of $E\times\langle u\rangle$'' are the unique two
central inductive involutions.  We denote the cyclic factor in that
statement by $C_0$, to avoid confusion with the elliptic curves used below.
Taken literally, the phrase is ambiguous, since
$C_0\times\langle u\rangle$ has three involutions $a,u,au$, where $a$ is
the unique involution of $C_0$.  The proof of
\cite[Theorem~4.3, pp.~277--278]{GK2015} removes the ambiguity: it shows
that $a$ and $au$ are inductive, whereas $u$ is the third central
involution and $X/\langle u\rangle$ has genus $1$.  Accordingly,
Theorem~\ref{thm:GKclassification} records the invariant content of the
published statement, and Proposition~\ref{prop:GK-refinement} identifies
the three involutions explicitly.
\end{remark}

\begin{proposition}\label{prop:no-genus3}
Case \textup{(ib)} of Theorem~\ref{thm:GKclassification} does not occur
for $g=3$.  Consequently $g-1=2^h\ge4$ in every instance of case
\textup{(ib)}.
\end{proposition}

\begin{proof}
In case \textup{(ib)}, let $a$ be the unique involution of the cyclic
subgroup $C_0$.  By Proposition~\ref{prop:GK-refinement}, the involution $a$ is
inductive.  Hence the induced group $S/\langle a\rangle$ satisfies
\textup{(I)}.  Since $|S|=4(g-1)$ in case \textup{(i)},
\[
        |S/\langle a\rangle|=2(g-1)\ge8,
\]
so $g\ge5$.  Thus case \textup{(ib)} cannot occur for $g=3$.

Finally, $S$ is a $2$-group and $|S|=4(g-1)$, so $g-1$ is a power of
$2$.  Therefore $g-1=2^h$ with $h\ge2$.
\end{proof}

\begin{remark}
For $g=3$ the two orders $4(g-1)$ and $2g+2$ are both equal to $8$, so
cases \textup{(i)} and \textup{(ii)} are not separated by the group
order alone.  The point is that case \textup{(ib)} contains an inductive
central involution, which forces the quotient group to satisfy
\textup{(I)} and hence rules out $g=3$.  Thus every curve in case
\textup{(ib)} has $g-1=2^h\ge4$, and Theorem~\ref{thm:converse} below
will cover all such curves.
\end{remark}

We shall be concerned with case \textup{(ib)}.  Giulietti and
Korchm\'aros gave isolated examples related to this case; see
\cite[Sections~6.2 and~6.5]{GK2015}.  In
Proposition~\ref{prop:sporadic-GK} below we identify these examples
directly with members of our family.  In his CoCoA 2015 lectures,
Korchm\'aros posed the following problem
\cite[Problem~11]{KorchmarosSlides}.

\begin{problem}[Korchm\'aros]\label{prob:Korchmaros11}
Construct an infinite family of curves of type \textup{(ib)}.
\end{problem}

A closely related Artin--Schreier construction over ordinary elliptic
curves already appears in \cite[Section~5]{GK2015}.  Giulietti and
Korchm\'aros start with an ordinary elliptic curve $\bar{\mathcal X}$ and,
for $n=2^h\ge8$, distinguish the subgroup
$\bar{\mathcal X}[n](k)$ from the nontrivial coset
$\bar{\mathcal X}[2n](k)\setminus\bar{\mathcal X}[n](k)$.  Their covers
are ramified over the latter set; the translation subgroup of order $n$
lifts non-splitly to a cyclic group of order $2n$, and the resulting
extremal group is dihedral; see in particular
\cite[Proposition~5.12 and Theorem~5.14]{GK2015}.  The construction below
uses instead the subgroup $E[q](k)$ itself and prescribes equal residues at
all its points.  This forces the lift of the translation group to split
from the deck involution and produces case \textup{(ib)} rather than the
dihedral case \textup{(ia)}.  The new ingredients here are this
split equal-residue realization, the descent argument giving the converse,
and the resulting parametrization of all curves of type \textup{(ib)}.

Artin--Schreier covers and their $p$-ranks have been studied since the
work of Subrao \cite{Subrao}.  Large automorphism groups of ordinary
curves in characteristic $2$ were studied, for example, by
Montanucci--Speziali \cite{MontanucciSpeziali}; their characteristic-$2$
bound is proved for ordinary curves of even genus, whereas the curves
constructed here have odd genus $2^h+1$.  Our concern is the extremal
configuration singled out in \cite{GK2015} and in
Problem~\ref{prob:Korchmaros11}.  The dihedral construction of
\cite[Section~5]{GK2015} is closely related to ours, but it has a non-split
lift and belongs to case \textup{(ia)}.  To the best of our knowledge, the
split equal-residue family below, the converse description of all curves of
type \textup{(ib)}, and its parameter space do not appear in those works.

We solve this problem by means of Artin--Schreier covers of ordinary
elliptic curves.  Let $q=2^h$ and let $E/k$ be ordinary.  The group
\[
        T=E[q](k)
\]
is cyclic of order $q$.  We choose a rational function $f$ with simple
poles at the points of $T$ and with equal nonzero residues with respect
to an invariant differential on $E$.  The translations by $T$ and the
elliptic involution lift to the Artin--Schreier cover
\[
        z^2+z=f.
\]
This gives the following result.

\begin{theorem}\label{thm:main}
Let $q=2^h$ with $h\ge3$, let $E/k$ be an ordinary elliptic curve, let
$\omega$ be a nonzero invariant differential on $E$, and let
$\lambda\in k^\times$.  Put $T=E[q](k)$.  Choose
$f\in H^0(E,\mathcal O_E(\sum_{P\in T}[P]))$ satisfying
\[
        \Res_P(f\omega)=\lambda\qquad(P\in T),
\]
and let $X$ be the smooth projective curve with function field
$k(E)(z)$, where $z^2+z=f$.  The $k$-isomorphism class of $X$ is
independent of the choice of $f$; we denote it by
$X(E,\omega,\lambda,q)$.  This curve has the following properties:
\begin{enumerate}[label=\textup{(\roman*)}]
\item $X$ is ordinary and bielliptic, with $g(X)=q+1$;
\item the deck involution of $X\to E$ is the unique bielliptic
involution of $X$;
\item
\[
        \Aut(X)\cong \Dih(C_q)\times C_2,
        \qquad
        |\Aut(X)|=4q=4(g(X)-1);
\]
\item $\Aut(X)$ has no global fixed point on $X$, and
$(X,\Aut(X))$ is of type \textup{(ib)};
\item there is a cyclic subgroup $C\cong C_q$ acting freely on $X$, and
$X/C$ is an ordinary curve of genus $2$.
\end{enumerate}
\end{theorem}

The Artin--Schreier construction used in the proof is also defined for
$q=4$; only the assertion on the full automorphism group in
Theorem~\ref{thm:main} requires $h\ge3$.  We use the same notation
$X(E,\omega,\lambda,4)$ below.  In particular, curves of type
\textup{(ib)} exist for every
\[
        g=2^h+1,\qquad h\ge2,
\]
where the case $h=2$ is established in Remark~\ref{rem:smallq} below.
We prove that every curve in case \textup{(ib)} arises from this
construction.  Proposition~\ref{prop:no-genus3} shows that no further
genus can occur.  For $h\ge3$ the converse, together with
Theorem~\ref{thm:main}, determines the full automorphism group.  For
$h=2$ the construction recovers every type \textup{(ib)} curve of genus
$5$, although we do not determine its full automorphism group.  For every
fixed $q\ge8$, the rigidity result below gives a complete two-parameter
classification of isomorphism classes: they are in bijection with
$(k^\times)^2$; see Theorem~\ref{thm:parameter-space}.

The paper is organized as follows.  Section~2 collects the facts on
Artin--Schreier covers and ordinary elliptic curves used in the sequel.
The construction and the full automorphism group are treated in
Section~3.  Section~4 gives the short orbits and ramification.
Section~5 proves the converse and studies the cyclic quotients.
Section~6 deals with the central involutions.  Explicit equations are
given in Section~7.

\section{Preliminaries}\label{sec:prelim}

We use the following standard facts.  If a $2$-group $H$ acts on a curve
$Y$ and $\Omega_1,\ldots,\Omega_r$ are its short orbits, the
Deuring--Shafarevich formula (see, for example, \cite{Subrao,NakajimaDS}) is
\begin{equation}\label{eq:DS}
 \gamma(Y)-1
 =|H|\bigl(\gamma(Y/H)-1\bigr)
  +\sum_{i=1}^r\bigl(|H|-|\Omega_i|\bigr).
\end{equation}
Nakajima's bound gives
\begin{equation}\label{eq:Nakajima}
        |H|\le4(\gamma(Y)-1)
\end{equation}
when $\gamma(Y)>1$ \cite[Theorem~1]{Nakajima}.  An ordinary curve is one for which
$\gamma(Y)=g(Y)$.

\begin{lemma}\label{lem:ordinary-weak}
Let a finite $2$-group $H$ act on an ordinary curve $Y$.  Then $Y/H$ is
ordinary and the action is weakly ramified.  In particular, every point
stabilizer $H_P$ is an elementary abelian $2$-group.  Compare
\cite{NakajimaDS}.
\end{lemma}

\begin{proof}
Let $Z=Y/H$.  Since $H$ is a $2$-group, at every ramified point
$H_P^{(0)}=H_P^{(1)}=H_P$, and hence
\[
 d_P\ge 2(|H_P|-1).
\]
Moreover,
\[
 \sum_{P\in Y}(|H_P|-1)
 =\sum_{\Omega\text{ short}}(|H|-|\Omega|).
\]
Riemann--Hurwitz therefore gives
\[
 2g(Y)-2\ge |H|(2g(Z)-2)
 +2\sum_{\Omega\text{ short}}(|H|-|\Omega|).
\]
On the other hand, doubling the Deuring--Shafarevich formula and using
$\gamma(Y)=g(Y)$ gives
\[
 2g(Y)-2=2|H|(\gamma(Z)-1)
 +2\sum_{\Omega\text{ short}}(|H|-|\Omega|).
\]
Since $\gamma(Z)\le g(Z)$, equality must hold throughout.  Thus
$\gamma(Z)=g(Z)$ and $d_P=2(|H_P|-1)$ at every ramified point, so
$H_P^{(2)}=1$.  Finally, the standard ramification quotient
$H_P^{(1)}/H_P^{(2)}$ is elementary abelian; hence so is $H_P$.
\end{proof}

If
\[
        k(X)=k(Y)(z),\qquad z^2+z=f,
\]
and $f$ has a simple pole at $P\in Y$, then the point above $P$ is
totally ramified, its lower jump is $1$, and its different exponent is
$2$; see \cite[Proposition~III.7.8]{Stichtenoth}.  We shall also use the
Castelnuovo--Severi inequality in the form given in
\cite[Theorem~III.10.3]{Stichtenoth}.

Let $E/k$ be an ordinary elliptic curve.  For every $h\ge1$,
\[
        E[2^h](k)\cong C_{2^h}.
\]
A nonzero invariant differential on $E$ has no zero.  Since $E$ is
ordinary in characteristic $2$, one has $j(E)\ne0$; hence
\[
        \Aut(E,0)=\{1,[-1]\}
\]
by \cite[Theorem~III.10.1]{Silverman}.  We write
\[
        \Dih(C_n)=\langle t,w\mid t^n=w^2=1,\;wtw=t^{-1}\rangle
\]
for the dihedral group of order $2n$.

The following lemma will be used repeatedly.

\begin{lemma}\label{lem:principalparts}
Let $T\subset E(k)$ be a finite subgroup of even order $q$, put
\[
        D=\sum_{P\in T}[P],
\]
and let $\omega$ be a nonzero invariant differential on $E$.  For every
$\lambda\in k$ there exists $f\in H^0(E,\mathcal O_E(D))$ such that
\[
        \Res_P(f\omega)=\lambda\qquad(P\in T).
\]
If $\lambda\ne0$, then the poles of $f$ are precisely the points of $T$
and they are simple.  Moreover, for every $Q\in T$,
\[
        \tau_Q^*f-f\in k,\qquad [-1]^*f-f\in k.
\]
\end{lemma}

\begin{proof}
Consider
\[
 \rho:H^0(E,\mathcal O_E(D))\longrightarrow k^T,\qquad
 f\longmapsto\bigl(\Res_P(f\omega)\bigr)_{P\in T}.
\]
Since $g(E)=1$ and $\deg D=q$, Riemann--Roch gives
$h^0(E,\mathcal O_E(D))=q$.  The kernel of $\rho$ consists of the
constants.  Indeed, every allowed pole is at most simple; hence, if all
residues vanish, every possible simple pole of $f\omega$ is removable.
Thus $f\omega$ is regular on $E$, and since $\omega$ is nowhere
vanishing, $f$ is constant.  Hence
$\dim\operatorname{im}\rho=q-1$.  By the residue theorem,
\[
 \operatorname{im}\rho
 \subseteq\left\{(a_P)_{P\in T}:\sum_{P\in T}a_P=0\right\},
\]
and the two spaces have the same dimension.  Since $q=0$ in $k$, the
constant vector $(\lambda)_{P\in T}$ belongs to this hyperplane.  Hence it
lies in the image of $\rho$, which proves the existence of $f$.

If $Q\in T$, translation by $Q$ preserves $\omega$ and permutes $T$.
For every $P\in T$ one has
\[
 \Res_P\bigl((\tau_Q^*f)\omega\bigr)
 =\Res_{P+Q}(f\omega)=\lambda=\Res_P(f\omega).
\]
Thus the possible simple poles of $(\tau_Q^*f-f)\omega$ have zero
residue and are removable.  Since $\omega$ is nowhere vanishing,
$\tau_Q^*f-f$ is regular, and hence constant.  The same argument applies
to $[-1]$: here $[-1](T)=T$ and $[-1]^*\omega=-\omega=\omega$ in
characteristic $2$, so the residues cancel in exactly the same way.
\end{proof}

\section{The construction}\label{sec:construction}

\subsection{The Artin--Schreier cover}

\begin{proof}[Proof of Theorem~\ref{thm:main}, except for the full automorphism group]
Because $E$ is ordinary,
\[
 E[2^h](k)\cong C_q.
\]
Let
\[
 T=E[2^h](k)=\langle P_0\rangle.
\]
Apply Lemma~\ref{lem:principalparts} with the prescribed
$\lambda\in k^\times$ to obtain $f\in k(E)$ having exactly the points of
$T$ as simple poles and satisfying
\[
 \Res_P(f\omega)=\lambda\qquad(P\in T).
\]

Let $X$ be the smooth projective curve with function field
\begin{equation}\label{eq:AS}
 k(X)=k(E)(z),\qquad z^2+z=f.
\end{equation}
The resulting $k$-isomorphism class is independent of the choice of $f$
with the prescribed residues.  Indeed, if $f'$ is another such function,
then $f'-f$ lies in the kernel of the residue map in
Lemma~\ref{lem:principalparts}, hence is a constant $c\in k$.  Since $k$ is
algebraically closed, there is $e\in k$ with $e^2+e=c$, and the change of
Artin--Schreier coordinate $z\mapsto z+e$ identifies the covers defined by
$f$ and $f'$.  Thus the notation $X(E,\omega,\lambda,q)$ is unambiguous
up to $k$-isomorphism.

The Artin--Schreier extension is nontrivial.  Indeed, a function of the
form $r^2+r+c$ cannot have a pole of odd order, while $f$ has simple
poles.  Thus $[k(X):k(E)]=2$.

At each $P\in T$ the function $f$ has a pole of order one.  Standard
Artin--Schreier ramification theory shows that the unique point of $X$
above $P$ is totally ramified, with lower jump $1$ and different exponent
$2$; see \cite[Proposition~III.7.8]{Stichtenoth}.  There is no other
ramification.  Since $g(E)=1$, Riemann--Hurwitz gives
\[
 2g(X)-2=\sum_{P\in T}2=2q,
\]
and hence
\begin{equation}\label{eq:genus}
 g(X)=q+1.
\end{equation}

Let $\gamma(Y)$ denote the $2$-rank of a curve $Y$.  The
Deuring--Shafarevich formula for the $C_2$-cover $X\to E$ gives
\[
 \gamma(X)-1
 =2\bigl(\gamma(E)-1\bigr)+q.
\]
Since $E$ is ordinary, $\gamma(E)=1$, and therefore
\[
 \gamma(X)=q+1=g(X).
\]
Thus $X$ is ordinary.

Let $\tau=\tau_{P_0}$.  By Lemma~\ref{lem:principalparts},
\[
 \tau^*f-f=c
\]
for some $c\in k$.  Choose $\alpha\in k$ with $\alpha^2+\alpha=c$.  Then
\[
 \widetilde\tau(R,z)=(R+P_0,z+\alpha)
\]
defines an automorphism of $X$.  Its projection to $E$ has exact order
$q$, and
\[
 \widetilde\tau^q(R,z)=(R,z+q\alpha)=(R,z),
\]
so $\widetilde\tau$ itself has exact order $q$.

The deck transformation is
\[
 u(R,z)=(R,z+1).
\]
Let $\iota=[-1]$ on $E$.  Write $\iota^*f-f=d$ and choose $b\in k$ with
$b^2+b=d$.  Then
\[
 w(R,z)=(-R,z+b)
\]
is well defined, since
\[
 (z+b)^2+(z+b)=f(R)+d=f(-R),
\]
and $w^2=1$ because $2b=0$.  Clearly
$[u,\widetilde\tau]=[u,w]=1$.  Moreover, using $w^{-1}=w$,
\[
\begin{aligned}
 w\widetilde\tau w^{-1}(R,z)
 &=w\widetilde\tau(-R,z+b)\\
 &=w(-R+P_0,z+b+\alpha)\\
 &=(R-P_0,z+\alpha)
 =\widetilde\tau^{-1}(R,z),
\end{aligned}
\]
where the last equality uses $-\alpha=\alpha$ in characteristic $2$.
Put $D=\langle\widetilde\tau,w\rangle$.  The projection $X\to E$
induces a surjection
\[
        D\longrightarrow
        \langle\tau_{P_0},[-1]\rangle\cong\Dih(C_q).
\]
The relations above give $|D|\le2q$, while the image has order $2q$.
Hence $D\cong\Dih(C_q)$.  The projection is faithful on $D$, so
$D\cap\langle u\rangle=1$.  Since $u$ is central,
\[
        S:=\langle\widetilde\tau,w,u\rangle
        \cong\Dih(C_q)\times C_2,
\]
and $|S|=4q$.  The quotient $X/\langle u\rangle$ is $E$, so $X$ is
bielliptic.  The subgroup $S$ has no global fixed point, because every
nontrivial power of $\widetilde\tau$ projects to a nontrivial translation
of $E$ and therefore has no fixed point.

In fact the cyclic subgroup
\[
 C:=\langle\widetilde\tau\rangle\cong C_q
\]
acts freely on $X$.  Put $Y=X/C$.  Since $X\to Y$ is an unramified cyclic
cover of degree $q$, Riemann--Hurwitz gives
\[
 2g(X)-2=q\bigl(2g(Y)-2\bigr).
\]
Using $g(X)=q+1$, we obtain $g(Y)=2$.  Likewise, the
Deuring--Shafarevich formula for the free $C$-action gives
\[
 \gamma(X)-1=q\bigl(\gamma(Y)-1\bigr).
\]
Since $\gamma(X)=q+1$, it follows that $\gamma(Y)=2=g(Y)$.  Hence $Y$ is
ordinary.  This proves part \textup{(v)} of the theorem at this stage.

Since $g(X)=q+1\ge9$, the Castelnuovo--Severi inequality implies that $X$
has at most one bielliptic involution: two distinct degree-$2$ maps to
genus-one curves would give
\[
 g(X)\le 2\cdot1+2\cdot1+(2-1)(2-1)=5;
\]
see \cite[Theorem~III.10.3]{Stichtenoth}.  Hence $u$ is the unique
bielliptic involution.
\end{proof}

\subsection{The full automorphism group}

\begin{proposition}\label{prop:fullaut}
For a curve $X=X(E,\omega,\lambda,q)$ of Theorem~\ref{thm:main},
\[
 \Aut(X)=S\cong \Dih(C_q)\times C_2.
\]
In particular, the group occurring in case \textup{(ib)} is the full
automorphism group of $X$.
\end{proposition}

\begin{proof}
The unique bielliptic involution $u$ is central in $\Aut(X)$: for any
$\alpha\in\Aut(X)$, the conjugate $\alpha u\alpha^{-1}$ is again a
bielliptic involution, hence equals $u$.  Therefore every automorphism of
$X$ descends through
\[
 \pi:X\longrightarrow E=X/\langle u\rangle.
\]
The descended automorphism preserves the branch locus of $\pi$, which is
precisely $T$.  The kernel of the descent homomorphism is the deck group
$\langle u\rangle$.  Thus
\begin{equation}\label{eq:descentinj}
 \Aut(X)/\langle u\rangle\hookrightarrow \Aut(E,T),
\end{equation}
where $\Aut(E,T)$ denotes the automorphisms of the curve $E$ preserving
$T$ setwise.

Every automorphism $\alpha$ of the genus-one curve $E$ can be written
uniquely in the form
\[
 \alpha=\tau_Q\circ\varphi,
 \qquad Q=\alpha(0),\quad \varphi\in\Aut(E,0).
\]
If $\alpha$ preserves $T$ setwise, then $Q=\alpha(0)\in T$.  Since $E$
is ordinary in characteristic $2$, one has $j(E)\ne0$ and
\[
 \Aut(E,0)=\{1,[-1]\}
\]
by \cite[Theorem~III.10.1]{Silverman}; both elements preserve
$T$.  Conversely every translation by a point of $T$, and $[-1]$, preserve
$T$.  Hence
\[
 \Aut(E,T)=T\rtimes\langle[-1]\rangle
 \cong\Dih(C_q).
\]
The subgroup $S/\langle u\rangle$ already maps isomorphically onto this
whole group.  The injection \eqref{eq:descentinj} therefore forces
\[
 \Aut(X)/\langle u\rangle=S/\langle u\rangle.
\]
Since $S\subseteq\Aut(X)$ and both contain $u$, it follows that
$\Aut(X)=S$.
\end{proof}

Proposition~\ref{prop:fullaut} completes the proof of Theorem~\ref{thm:main}.  In particular,
\[
 |\Aut(X)|=4q=4(g(X)-1).
\]
Since $q\ge8$, one has $|\Aut(X)|\ge8$ and
$|\Aut(X)|=4q>2q=2(g(X)-1)$, so hypothesis \textup{(I)} holds; hypothesis
\textup{(II)} was proved above.  Hence the classification theorem applies.  Case \textup{(ii)} is
impossible because $|\Aut(X)|=4(g-1)$ whereas case \textup{(ii)} has
order $2g+2$; these are equal only for $g=3$, while here $g\ge9$.
Case \textup{(iii)} is impossible because the central involution $u$
has quotient $E$ of genus $1$ and hence is not inductive.  Finally the group is not dihedral: its center has order $4$, whereas
a dihedral group of the same order $4q$ has center of order $2$.
Therefore case \textup{(ia)} is excluded.  Thus
$(X,\Aut(X))$ is exactly of type \textup{(ib)}.

\begin{corollary}\label{cor:problem11}
Case \textup{(ib)} occurs exactly in the genera
\[
        g=2^h+1,\qquad h\ge2.
\]
In particular, there are infinitely many genera for which curves of type
\textup{(ib)} exist.
\end{corollary}

\begin{proof}
Necessity is Proposition~\ref{prop:no-genus3}.  For $h\ge3$ existence is
Theorem~\ref{thm:main}, and the case $h=2$ is given by
Remark~\ref{rem:smallq}.
\end{proof}

\section{Ramification and the group action}\label{sec:shortorbits}

\subsection{Short orbits and ramification}

$X$ is ordinary, so $\gamma(X)=g(X)=q+1$, and therefore
\[
 |\Aut(X)|=4q=4(\gamma(X)-1).
\]
Thus the full automorphism group is itself a $2$-group attaining
Nakajima's bound.  The orbit structure can also be read off completely.

\begin{proposition}\label{prop:shortorbits}
Let $G=\Aut(X)$.  Then
\[
 X/G\cong \mathbf P^1.
\]
Moreover, $G$ has exactly two short orbits $\Omega_1,\Omega_2$, with
\[
 |\Omega_1|=2q,\qquad |\Omega_2|=q.
\]
The second orbit is the ramification locus of the Artin--Schreier map
$\pi:X\to E$; its point stabilizers are isomorphic to $C_2\times C_2$.
The stabilizer of a point in the first short orbit is cyclic of order $2$.
All remaining $G$-orbits are long.
\end{proposition}

\begin{proof}
Write $T=\langle P_0\rangle\cong C_q$.  Since $u$ is central and
$X/\langle u\rangle=E$, we have
\[
 X/G\cong E/(T\rtimes\langle[-1]\rangle).
\]
The quotient $E/T$ is again an elliptic curve, and $[-1]$ descends to its
elliptic involution.  Hence
\[
 E/(T\rtimes\langle[-1]\rangle)
 \cong (E/T)/\langle[-1]\rangle\cong\mathbf P^1.
\]

Let $\Omega_2$ be the set of ramification points of $\pi$.  There is one
point of $X$ above each point of $T$, so $|\Omega_2|=q$.  The translation
subgroup acts transitively on this set; hence $\Omega_2$ is a $G$-orbit.
Its stabilizer has order $|G|/q=4$.  Since the cyclic translation subgroup
acts freely on $X$, the stabilizer meets it trivially; it contains the
deck involution $u$.  If $w_P$ denotes the reflection fixing the
corresponding point $P\in T$ (its projection to $E$ is
$\tau_{2P}\circ[-1]$), then the stabilizer is
$\{1,u,w_P,uw_P\}$.  Since $u$ and $w_P$ are commuting involutions, this
stabilizer is a Klein four group.

Now apply the Deuring--Shafarevich formula to the full $2$-group $G$.
Since $X$ is ordinary and $X/G$ is rational,
\[
 q=\gamma(X)-1
   =-4q+\sum_{\Omega\text{ short}}(4q-|\Omega|).
\]
The Deuring--Shafarevich identity therefore requires the total
short-orbit contribution to be $5q$.  The known orbit $\Omega_2$
contributes $3q$, so exactly $2q$ remains.  On the other hand, because the
cyclic subgroup $C_q$ acts freely, every point stabilizer injects into
$G/C_q$, which has order $4$.  Thus every further short orbit has length
$q$ or $2q$.  A further orbit of length $q$ would already contribute
$3q$, which is too much.  The remaining contribution is exactly $2q$,
so there is precisely one further short orbit, it has length $2q$, and its
stabilizer has order $2$.
\end{proof}

\begin{corollary}\label{cor:weakram}
The action of $G$ on $X$ is weakly ramified.  More explicitly, if
$P\in\Omega_2$, then
\[
 G_P^{(0)}=G_P^{(1)}\cong C_2\times C_2,
 \qquad G_P^{(2)}=1,
\]
and the different exponent at $P$ is $6$.  If $P\in\Omega_1$, then
\[
 G_P^{(0)}=G_P^{(1)}\cong C_2,
 \qquad G_P^{(2)}=1,
\]
and the different exponent at $P$ is $2$.  Consequently the total
different of $X\to X/G$ has degree $10q$.
\end{corollary}

\begin{proof}
By Lemma~\ref{lem:ordinary-weak}, the action is weakly ramified.  Hence
$G_P^{(2)}=1$, while for a ramified point of a $2$-group action one has
$G_P^{(0)}=G_P^{(1)}=G_P$.  Proposition~\ref{prop:shortorbits} therefore
gives local different exponent
\[
        2(|G_P|-1),
\]
which is $6$ on $\Omega_2$ and $2$ on $\Omega_1$.  Thus
\[
 \deg\operatorname{Diff}(X/(X/G))
   =q\cdot6+2q\cdot2=10q.
\]
This also checks Riemann--Hurwitz:
\[
 2g(X)-2=2q=-2|G|+10q=-8q+10q.
\]
\end{proof}

\section{The converse and cyclic quotients}\label{sec:families}

\subsection{Isomorphisms in the family}

The parameter occurring in the construction is intrinsic.

\begin{proposition}\label{prop:parameters}
Let $q=2^h$ with $h\ge3$.  Let $E_i/k$ be ordinary elliptic curves and
let $\omega_i$ be nonzero invariant differentials.  Then
\[
 X(E_1,\omega_1,\lambda_1,q)\cong
 X(E_2,\omega_2,\lambda_2,q)
\]
if and only if there is an isomorphism of elliptic curves
$\varphi:E_1\to E_2$ such that
\[
        \varphi^*(\lambda_2^{-1}\omega_2)
        =\lambda_1^{-1}\omega_1.
\]
In particular,
\[
        X(E,\omega,\lambda,q)
        \cong X(E,\lambda^{-1}\omega,1,q).
\]
Thus, after normalizing the residue to $1$, the curves are parametrized
by isomorphism classes of pairs $(E,\eta)$, where $E$ is ordinary and
$0\ne\eta\in H^0(E,\Omega_E^1)$.
\end{proposition}

\begin{proof}
The last displayed isomorphism follows at once from
\[
        \Res_P(f\omega)=\lambda
        \quad\Longleftrightarrow\quad
        \Res_P(f\,\lambda^{-1}\omega)=1.
\]
Suppose that
\[
 \Phi:X(E_1,\eta_1,1,q)\longrightarrow X(E_2,\eta_2,1,q)
\]
is an isomorphism.  The bielliptic involution is unique, so $\Phi$
descends to an isomorphism $\psi:E_1\to E_2$ carrying
$E_1[q](k)$ onto $E_2[q](k)$.  In particular $\psi(0)\in E_2[q](k)$.
Translation by $-\psi(0)$ preserves the branch set and lifts to the
second cover.  After composing $\Phi$ with such a lift, we may therefore
assume that the descended map preserves the origins; denote it by
$\varphi:E_1\to E_2$.  Write
\[
        \varphi^*\eta_2=c\eta_1,\qquad c\in k^\times.
\]
Let $f_i$ be Artin--Schreier functions defining the two covers.  An
isomorphism of quadratic Artin--Schreier extensions over $\varphi$
sends an Artin--Schreier generator to the other generator plus an
element of the base field.  Hence
\[
        \varphi^*f_2-f_1=r^2+r
\]
for some $r\in k(E_1)$.  The left-hand side has at most simple poles,
whereas, if $r$ has a pole of order $m>0$, then $r^2+r$ has a pole of
order $2m$.  Thus $r$ has no pole and is constant.  For
$P\in E_1[q](k)$ we then have
\[
 \Res_P(\varphi^*f_2\,\eta_1)
   =c^{-1}\Res_{\varphi(P)}(f_2\eta_2)=c^{-1},
 \qquad
 \Res_P(f_1\eta_1)=1.
\]
Since the difference has zero residue, $c^{-1}=1$, and therefore
$c=1$.

Conversely, suppose that $\varphi:E_1\to E_2$ preserves origins and
satisfies $\varphi^*\eta_2=\eta_1$.  Then $\varphi$ carries
$E_1[q](k)$ onto $E_2[q](k)$, and $\varphi^*f_2$ and $f_1$ have the same
residues at every point of $E_1[q](k)$.  By
Lemma~\ref{lem:principalparts} their difference is constant; since $k$
is algebraically closed, that constant is of the form $e^2+e$.  The
change of Artin--Schreier coordinate $z\mapsto z+e$ gives the required
isomorphism.
\end{proof}

\begin{corollary}\label{cor:intrinsic-parameter}
For fixed $q\ge8$, after normalizing the residue to $1$, the isomorphism
classes obtained from the construction are in bijection with the
isomorphism classes of pairs $(E,\eta)$, where $E$ is ordinary elliptic
and $0\ne\eta\in H^0(E,\Omega_E^1)$.
\end{corollary}

\subsection{Artin--Schreier descent and the converse}

The residue construction is not merely a source of examples.  It
recovers every curve occurring in case \textup{(ib)}.

\begin{lemma}[Local Artin--Schreier residue criterion]\label{lem:localASresidue}
Let $K=k((t))$, where $k$ is algebraically closed of characteristic $2$,
and let $\vartheta$ be a regular differential which is nonzero at
$t=0$.  Let $g_1,g_2\in K$ have at most a simple pole.  If
\[
        \Res(g_1\vartheta)=\Res(g_2\vartheta),
\]
then $g_1-g_2\in k[[t]]$.  In particular the Artin--Schreier equation
\[
        y^2+y=g_1-g_2
\]
is unramified at $t=0$.
\end{lemma}

\begin{proof}
Write $\vartheta=v(t)\,dt$ with $v(0)\ne0$.  The polar part of
$g_1-g_2$ is $c/t$ for some $c\in k$, and the residue of
$(g_1-g_2)\vartheta$ is $cv(0)$.  Hence $c=0$.  Thus
$g_1-g_2\in k[[t]]$, and the last assertion follows since the derivative
of $y^2+y-(g_1-g_2)$ with respect to $y$ is $1$.
\end{proof}

\begin{ASdescent}\label{lem:ASdescent}
Let $k$ be an algebraically closed field of characteristic $2$, let $E/k$
be an ordinary elliptic curve, and let
\[
        T=E[q](k)\cong C_q,\qquad q=2^h,\quad h\ge1.
\]
Write $\phi:E\to F:=E/T$ for the quotient isogeny.  Let
$\pi:Y\to E$ be a connected Artin--Schreier double cover with deck
involution $u$.  Suppose that the branch locus of $\pi$ is exactly $T$,
with lower jump $1$ at every branch point, and that the translations by
$T$ lift to a subgroup $\widetilde T\leq\Aut(Y)$ such that
\[
        \widetilde T\cong T,\qquad
        \widetilde T\cap\langle u\rangle=1,
        \qquad [\widetilde T,u]=1.
\]
Then, for every nonzero invariant differential $\omega$ on $E$, there is
a unique $\lambda\in k^\times$ such that
\[
        Y\cong X(E,\omega,\lambda,q)
\]
as double covers of $E$.
\end{ASdescent}

\begin{proof}
Put $Z=Y/\widetilde T$.  Since $\widetilde T$ commutes with $u$ and has
trivial intersection with $\langle u\rangle$, the subgroup
$\widetilde T\times\langle u\rangle$ has order $2q$, and its fixed field
is $k(F)$.  Inside $k(Y)$ we have
\[
 k(Z)=k(Y)^{\widetilde T},\qquad k(E)=k(Y)^{\langle u\rangle}.
\]
Moreover,
\[
 k(Z)\cap k(E)
   =k(Y)^{\langle\widetilde T,u\rangle}=k(F).
\]
Consequently $k(Y)=k(Z)k(E)$ over $k(F)$, so $Y$ is the normalization of
$Z\times_F E$.

The isogeny $\phi:E\to F$ is finite \emph{\'etale}: its kernel is the reduced
constant subgroup scheme determined by the $q$ distinct points of $T$.
The pullback of
$Z\to F$ along $\phi$ is $Y\to E$.  Since the latter is ramified exactly
at the fibre $T=\phi^{-1}(O_F)$, the cover $Z\to F$ is ramified exactly
at $O_F$.  Ramification groups are unchanged by unramified extension of
complete local fields; hence the lower jump of $Z/F$ at $O_F$ is again
$1$.

Let $\theta$ be the unique invariant differential on $F$ satisfying
$\phi^*\theta=\omega$.  Choose a reduced local Artin--Schreier equation
for $Z/F$ at $O_F$,
\[
        y^2+y=\frac{a}{t}+b(t),\qquad a\ne0,\quad b(t)\in k[[t]],
\]
and define
\[
        \lambda=\Res_{O_F}\!\left(\frac{a}{t}\theta\right).
\]
Because the pole is simple, changing the reduced Artin--Schreier
representative by $r^2+r$ cannot alter its simple polar coefficient;
thus $\lambda$ is independent of the chosen reduced equation and is
nonzero.

Let $X_\lambda=X(E,\omega,\lambda,q)$, and let
$Z_\lambda=X_\lambda/C_\lambda$, where $C_\lambda\cong T$ is the cyclic
subgroup lifting the translations by $T$.  By the same quotient and
unramified-base-change argument, $Z_\lambda\to F$ is ramified only at
$O_F$, has lower jump $1$, and has local residue $\lambda$ there.
Let $\xi,\xi_\lambda\in k(F)/\wp(k(F))$ be the global
Artin--Schreier classes of $Z/F$ and $Z_\lambda/F$.  At $O_F$, choose
reduced representatives of their images in the completed local field.
They have at most a simple pole and the same residue with respect to
$\theta$; hence Lemma~\ref{lem:localASresidue} shows that the local
class of $\xi-\xi_\lambda$ is unramified at $O_F$.  At every other
place both local classes are unramified, and so is their difference.
Thus $\xi-\xi_\lambda$ is an everywhere-unramified quadratic
Artin--Schreier class on $F$.

Because $F$ is an ordinary elliptic curve, its $2$-rank is $1$;
equivalently, the $\mathbf F_2$-vector space of \emph{\'etale} $C_2$-torsors on
$F$ has dimension $1$.  Hence $F$ has exactly one nontrivial connected
\emph{\'etale} $C_2$-cover.  Let
$T'\le T$ be the subgroup of order $q/2$ (for $q=2$ take $T'=0$).  The
separable isogeny
\[
        E/T'\longrightarrow E/T=F
\]
is a nontrivial \emph{\'etale} double cover, and therefore represents this unique
nonzero unramified Artin--Schreier class.  After pullback along
$\phi:E\to F$ it becomes trivial: the factorization
\[
        E\longrightarrow E/T'\longrightarrow F
\]
defines a section of the pulled-back $C_2$-torsor.  Thus the possible
unramified difference between $Z$ and $Z_\lambda$ disappears after base
change to $E$.  Since the respective pullbacks are $Y$ and $X_\lambda$,
we obtain
\[
        Y\cong X(E,\omega,\lambda,q)
\]
as double covers of $E$.

Finally suppose that the same double cover of $E$ were obtained from
parameters $\lambda$ and $\lambda'$.  If $f,f'$ are corresponding
Artin--Schreier functions, an isomorphism over $E$ gives
$f'-f=r^2+r$ for some $r\in k(E)$.  The left-hand side has at most
simple poles.  Therefore $r$ has no pole, so $r$ is constant, and
$f'-f$ has zero residue at every point of $T$.  Hence
$\lambda'=\lambda$.  This proves uniqueness.
\end{proof}

\begin{remark}\label{rem:descent-sharp}
The condition
$\widetilde T\cap\langle u\rangle=1$ in the descent lemma is essential.
Without it, the inverse image of the translation group need not split as
$C_q\times C_2$: the non-split possibility is cyclic of order $2q$, with
a lift $\widetilde t$ satisfying $\widetilde t^{q}=u$.  This is the
extension pattern produced by the nontrivial \emph{\'etale} twist in the
descent picture.  It is the mechanism exhibited by the dihedral
case-\textup{(ia)} quotient $X^+$ in
Theorem~\ref{thm:centralquotients} below.
\end{remark}

\begin{theorem}[Converse to the construction]\label{thm:converse}
Let $Y/k$ be a smooth projective curve over an algebraically closed field
$k$ of characteristic $2$, and suppose that $Y$ admits a $2$-subgroup
$S\leq\Aut(Y)$ satisfying \textup{(I)} and \textup{(II)} such that
$(Y,S)$ is in case \textup{(ib)} of
Theorem~\ref{thm:GKclassification}.  Put $q=g(Y)-1$.  Then
$q=2^h$ for some $h\ge2$, and there are an ordinary elliptic curve $E/k$
and a nonzero invariant differential $\eta$ on $E$ such that
\[
        Y\cong X(E,\eta,1,q).
\]
If $h\ge3$, then
\[
        S=\Aut(Y)\cong \Dih(C_q)\times C_2.
\]
If $h=2$, then $S\cong\Dih(C_4)\times C_2$ and $S$ is a Sylow
$2$-subgroup of $\Aut(Y)$.
\end{theorem}

\begin{proof}
By Proposition~\ref{prop:no-genus3}, case \textup{(ib)} forces
$q=g(Y)-1=2^h$ with $h\ge2$.  Write, as in case \textup{(ib)} of
\cite[Theorem~1.2]{GK2015},
\[
        S=(C\times\langle u\rangle)\rtimes\langle w\rangle,
        \qquad C\cong C_q,
\]
where $S/\langle u\rangle$ is dihedral of order $2q$.  By
Proposition~\ref{prop:GK-refinement}, $u$ is the central non-inductive
involution.  The classification also
gives
\[
        \gamma(Y)=g(Y)=q+1.
\]
Put
\[
        Z:=Y/\langle u\rangle
\]
and write $g_0=g(Z)$ and $\gamma_0=\gamma(Z)$.  Let $r$ be the
number of fixed points of $u$.  The Deuring--Shafarevich formula gives
\begin{equation}\label{eq:converse-DS}
        q=2(\gamma_0-1)+r.
\end{equation}
On the other hand, Riemann--Hurwitz gives
\[
        2q=2(2g_0-2)+\deg\operatorname{Diff}(Y/Z).
\]
Every fixed point of an involution in characteristic $2$ contributes at
least $2$ to the different.  Hence, using \eqref{eq:converse-DS},
\[
\begin{aligned}
        2q
        &\ge 4g_0-4+2r\\
        &=2q+4(g_0-\gamma_0).
\end{aligned}
\]
Since always $\gamma_0\le g_0$, equality must hold throughout.  Thus
$Z$ is ordinary and every ramified point of $Y\to Z$ has different
exponent $2$, equivalently lower jump $1$.

We claim that $g_0=1$.  First, $g_0\ne0$.  Indeed,
$S/\langle u\rangle$ acts faithfully on $Z$, and it contains the image
of $C$, a cyclic group of order $q\ge4$.  But in characteristic $2$ every nontrivial element of
$2$-power order in $\operatorname{PGL}_2(k)=\Aut(\mathbf P^1)$ is
unipotent, and hence has order $2$.  Thus a cyclic subgroup of order
$q\ge4$ cannot act faithfully on a rational curve.

Suppose that $g_0\ge2$.  Since $Z$ is ordinary,
\eqref{eq:converse-DS} gives
\[
        2(g_0-1)=q-r\le q,
\]
and therefore the induced group $S/\langle u\rangle$, of order $2q$,
still satisfies hypothesis \textup{(I)}.  Since $u$ is non-inductive,
hypothesis \textup{(II)} must fail; hence $S/\langle u\rangle$ fixes a
point of $Z$.  By Lemma~\ref{lem:ordinary-weak}, its point stabilizer is
elementary abelian.  This is impossible because it contains the image of
a generator of $C$, which has order $q\ge4$.  Hence $g_0=1$.

It follows from \eqref{eq:converse-DS} that $\gamma_0=1$ and $r=q$.
Thus $Z$ is an ordinary elliptic curve, the cover $Y\to Z$ is
branched at exactly $q$ points, and the lower jump is $1$ at every branch
point.

The subgroup $C$ descends faithfully to $Z$.  Choose temporarily an
origin $O\in Z(k)$ and a generator $c\in C$.  Since $Z$ is ordinary in
characteristic $2$, $\Aut(Z,O)=\{1,[-1]\}$.  Writing the descended
automorphism as a translation followed by an element of $\Aut(Z,O)$,
its linear part
cannot be $[-1]$, because every automorphism of the form
$\tau_Q\circ[-1]$ is an involution.  The image of $c$, which has order
$q\ge4$, is therefore translation by a point of exact order $q$.  Hence
$C$ acts on $Z$ as translations by a cyclic subgroup $T$ of order $q$.

Let $B\subset Z$ be the branch locus.  It is $C$-invariant.  Translation
by $T$ is free on $Z$, and $|B|=|T|=q$, so $B=Q+T$ for some
$Q\in Z(k)$.  Use $Q$ as the origin of the genus-one curve $Z$, and
denote the resulting elliptic curve by $E$.  If $+$ denotes the old group
law, the new group law is
\[
        P\oplus P'=P+P'-Q.
\]
Thus the old translation $P\mapsto P+t$ is translation, for the new
group law, by the point $Q+t$, and
\[
        (Q+t_1)\oplus(Q+t_2)=Q+(t_1+t_2).
\]
Consequently $B=Q+T$ is a cyclic subgroup of order $q$ in the new group
law, and every element of $B$ is killed by $q$.  Since an ordinary
elliptic curve in characteristic $2$ has exactly $q=2^h$ geometric
$q$-torsion points, it follows that
\[
        B=E[q](k),
\]
and the descended action of $C$ is exactly translation by $E[q](k)$.

Choose a nonzero invariant differential $\omega$ on $E$.  The double
cover $Y\to E$ is branched exactly over $E[q](k)$, with lower jump $1$ at
every branch point.  Moreover, $C$ commutes with $u$ and
$C\cap\langle u\rangle=1$.  Lemma~\ref{lem:ASdescent} therefore
gives a unique $\lambda\in k^\times$ such that
\[
        Y\cong X(E,\omega,\lambda,q).
\]
Putting $\eta=\lambda^{-1}\omega$ simply normalizes all residues to
$1$, and hence
\[
        Y\cong X(E,\eta,1,q).
\]

If $h\ge3$, Proposition~\ref{prop:fullaut} gives
\[
        \Aut(Y)\cong\Dih(C_q)\times C_2
\]
of order $4q$.  Since $|S|=4q$, we obtain $S=\Aut(Y)$.

It remains to consider $h=2$.  The construction itself gives a subgroup
\[
        S_0\cong\Dih(C_4)\times C_2
\]
of $\Aut(Y)$ of order $16$.  Since $Y$ is ordinary of genus $5$, Nakajima's bound, applied to an
arbitrary $2$-subgroup of $\Aut(Y)$, shows that every such subgroup has
order at most $4(\gamma(Y)-1)=16$.  Thus both $S$ and $S_0$ are Sylow $2$-subgroups of
$\Aut(Y)$.  They are conjugate, and in particular
$S\cong\Dih(C_4)\times C_2$.
\end{proof}

\begin{corollary}\label{cor:ib-group}
Let $(Y,S)$ be in case \textup{(ib)} and put $q=g(Y)-1$.  Then
\[
        S\cong\Dih(C_q)\times C_2.
\]
In particular, the other abstract group compatible with the structural
description in Theorem~\ref{thm:GKclassification}, namely
\[
 \left\langle c,u,w\ \middle|\
 c^q=u^2=w^2=1,\ [c,u]=[u,w]=1,\ wcw=c^{-1}u
 \right\rangle,
\]
never occurs in case \textup{(ib)}.
\end{corollary}

\begin{proof}
By Proposition~\ref{prop:no-genus3}, one has $q=2^h$ with $h\ge2$.
For $h\ge3$ this follows from Theorem~\ref{thm:converse}, which gives
$S=\Aut(Y)\cong\Dih(C_q)\times C_2$.  For $h=2$, the same theorem
shows directly that $S\cong\Dih(C_4)\times C_2$.

For completeness, the alternative displayed in the statement is genuinely
different from the split group.  If $A=C_q\times\langle u\rangle$ and
$c$ generates $C_q$, an involutory automorphism of $A$ fixing $u$ and
inducing inversion on $A/\langle u\rangle$ sends $c$ either to $c^{-1}$
or to $c^{-1}u$.  These give the split and twisted presentations,
respectively.  In the twisted case
\[
        (c^iu^jw)^2=u^i.
\]
Hence there are $q$ involutions outside $A$, whereas
$\Dih(C_q)\times C_2$ has $2q$ involutions outside $A$.  The two groups
are therefore not isomorphic.
\end{proof}

\begin{proposition}[The examples of Giulietti--Korchm\'aros]\label{prop:sporadic-GK}
\leavevmode
\begin{enumerate}[label=\textup{(\alph*)}]
\item \emph{Genus $9$.}  Let $\mu\in\F_{16}$ satisfy $\mu^4+\mu+1=0$, and
let $\mathcal X$ be the smooth projective model of the plane curve of
\cite[Section~6.2]{GK2015},
\begin{equation}\label{eq:GK62}
 (x^7+\mu^5x^5+\mu^{13}x^3+\mu^9x)(z^4+z)
 =x^8+\mu^2x^6+\mu^8x^4+\mu^3x^2+\mu^2 .
\end{equation}
Let $E_\mu$ be the elliptic curve $y^2+xy=x^3+\mu$. Then
\[
 \mathcal X\cong X(E_\mu,dx/x,1,8).
\]
In particular $\mathcal X$ is ordinary of genus $9$, and
$\Aut(\mathcal X)\cong\Dih(C_8)\times C_2$. The pair
$(\mathcal X,\Aut(\mathcal X))$ is of type \textup{(ib)}. In the
parametrization of Section~\ref{sec:families} it corresponds to
$(A,c)=(\mu,1)$.

\item \emph{Genus $5$.}  Let $\mathcal Y$ be the Artin--Mumford curve
\[
        (x^4+x)(y^4+y)=1,
\]
used in \cite[Section~6.4]{GK2015}; for background on the Artin--Mumford
family see also \cite{MontanucciZini}.  For
$\alpha,\beta\in\F_4$ let $\varphi_{\alpha,\beta}(x,y)=(x+\alpha,y+\beta)$,
and let $\rho(x,y)=(y,x)$. Put
\[
 S=\{\varphi_{\alpha,\beta}\}\rtimes\langle\rho\rangle,\qquad
 u=\varphi_{1,1},\qquad
 \bar{\mathcal Y}=\mathcal Y/\langle u\rangle,\qquad \bar S=S/\langle u\rangle .
\]
Then $(\bar{\mathcal Y},\bar S)$ is of type \textup{(ib)}, and
$\bar{\mathcal Y}\cong X(E,\eta,1,4)$ for some pair $(E,\eta)$. The
elliptic curve $E$ is the smooth curve $(s^2+s)(t^2+t)=1$ of bidegree
$(2,2)$ in $\mathbf P^1\times\mathbf P^1$.
\end{enumerate}
\end{proposition}

\begin{proof}
\textbf{(a)}  Throughout we use the following identities in $\F_{16}$:
\[
 \mu^4=\mu+1,\quad \mu^5=\mu^2+\mu,\quad \mu^8=\mu^2+1,\quad
 \mu^9=\mu^3+\mu,\quad \mu^{10}=\mu^2+\mu+1,
\]
\[
 \mu^{13}=\mu^3+\mu^2+1,\quad \mu^{14}=\mu^3+1,\quad \mu^{15}=1 .
\]
Since the Frobenius map $a\mapsto a^2$ is bijective on $\F_{16}$, we
have $\sqrt{\mu}=\mu^8$, $\mu^{1/4}=\mu^4$ and $\mu^{1/8}=\mu^2$.

\smallskip
\emph{Step 1: the $x$-coordinates of $E_\mu[8]$.}
Take $E_A:y^2+xy=x^3+A$ with $A\in k^\times$. The tangent at $R=(x,y)$
has slope $x+y/x$. Using $(y/x)^2+y/x=x+A/x^2$ we get
\[
 x(2R)=\frac{x^4+A}{x^2}.
\]
The nonzero $2$-torsion point is $T_2=(0,\sqrt A)$.

A point $R$ has order $4$ iff $x(2R)=0$, i.e.\ iff $x^4=A$, i.e.\ iff
$x=A^{1/4}$.

A point $R$ has order $8$ iff $x(2R)=A^{1/4}$, i.e.\ iff
\[
 x^4+A^{1/4}x^2+A=(x^2+A^{1/8}x+A^{1/2})^2=0 .
\]
Hence the $x$-coordinates of the six points of
$E_A[8](k)\setminus\{O,T_2\}$ are exactly the roots of
\[
 P_A(x)=(x+A^{1/4})(x^2+A^{1/8}x+A^{1/2}).
\]
These are three distinct nonzero roots. The quadratic factor is
separable, and substituting $x=A^{1/4}$ into it gives $A^{3/8}\neq0$.
Each root carries two points of $E_A$.

For $A=\mu$ this gives
\[
 P:=P_\mu=x^3+\mu^{10}x^2+\mu^{14}x+\mu^{12},
 \qquad
 P'=x^2+\mu^{14}.
\]
Squaring coefficient by coefficient,
\[
 xP^2=x^7+\mu^5x^5+\mu^{13}x^3+\mu^9x .
\]
This is the left-hand coefficient in \eqref{eq:GK62}.

\smallskip
\emph{Step 2: two polynomial identities.}  Put
\[
 Q=x^4+\mu x^3+\mu^4x^2+\mu^9x+\mu,
 \qquad
 M=\mu x^3+\mu^3x^2+\mu^2x+\mu^5 .
\]
Then $Q^2=x^8+\mu^2x^6+\mu^8x^4+\mu^3x^2+\mu^2$, which is the
right-hand side of \eqref{eq:GK62}. We claim
\begin{align}
 M^2+xPM&=(x^3+\mu)P^2+xQ^2, \label{eq:star}\\
 M+x^2P'&=(x+\mu^8)\,P. \label{eq:starstar}
\end{align}
Both sides of \eqref{eq:star} expand to
\[
 \mu x^7+\mu x^6+\mu^3x^5+\mu^6x^4+\mu x^3+\mu^{14}x^2+\mu^2x+\mu^{10}.
\]
The terms $x^9$ cancel on the right. Identity \eqref{eq:starstar} is a
direct expansion.

Two consequences will be used below. By \eqref{eq:starstar},
$M(a)=a^2P'(a)\neq0$ for every root $a$ of $P$. Moreover
$M(0)=\mu^5=\mu^8\cdot\mu^{12}=\sqrt\mu\,P(0)$.

\smallskip
\emph{Step 3: an Artin--Schreier function of the construction.}
On $E_\mu$ put
\[
 f=\frac yx+\frac{M(x)}{xP(x)},\qquad \omega=\frac{dx}{x}.
\]
We check the poles and residues of $f\omega$ at each point.

\emph{Points over a root $a$ of $P$.} Here $x-a$ is a local parameter,
and $y/x$ is regular. The term $M/(xP)$ has a simple pole, and
\[
 \Res(f\omega)=\frac{M(a)}{a^2P'(a)}=1 .
\]

\emph{The point $T_2=(0,\mu^8)$.} Put $s=y+\mu^8$. The equation becomes
$s^2=x(s+\mu^8+x^2)$. Hence $s$ is a local parameter,
$\operatorname{ord}(x)=2$, and
\[
 \omega=\frac{ds}{\mu^8+s+x^2},\qquad
 \frac1x=\frac{\mu^8+s+x^2}{s^2}.
\]
Since $M(0)/P(0)=\mu^8$, we can write
$y+M/P=s+x\,r(x)$ with $r$ regular at $x=0$. Therefore
\[
 f\omega=\frac{s+x\,r(x)}{s^2}\,ds
 =\Bigl(\frac1s+\frac{r(x)}{\mu^8+s+x^2}\Bigr)ds .
\]
So $f$ has a simple pole at $T_2$, and the residue is $1$.

\emph{The point $O$.} Here $y/x$ has a simple pole, and $M/(xP)$ is
regular because $\deg M=3<4$. The residue theorem then gives
$\Res_O(f\omega)=7\cdot1=1$.

\emph{All other points.} $f$ is regular there.

By Step~1, $f$ has simple poles exactly on $E_\mu[8](k)$, with all
residues equal to $1$ with respect to $dx/x$. By the independence of the
choice of $f$ (proof of Theorem~\ref{thm:main}), the cover $z^2+z=f$ is
$X(E_\mu,dx/x,1,8)$.

\smallskip
\emph{Step 4: the plane model.}  Put $W=z^2+z=f$. Then
\[
 y=xW+M/P .
\]
Substitute this into $y^2+xy=x^3+\mu$ and use \eqref{eq:star}:
\[
 x^2(W^2+W)=x^3+\mu+\frac{M^2+xPM}{P^2}=\frac{xQ^2}{P^2}.
\]
Since $W^2+W=z^4+z$, this is exactly
$xP^2(z^4+z)=Q^2$, i.e.\ \eqref{eq:GK62}.

Conversely, $y\in k(x,z)$, so $k(x,z)=k(E_\mu)(z)$. This field has
degree $4$ over $k(x)$. Hence \eqref{eq:GK62} is irreducible, and its
smooth model is $X(E_\mu,dx/x,1,8)$. Theorem~\ref{thm:main} now gives all remaining assertions.

\medskip
\textbf{(b)}

\emph{Genus and smoothness.} The affine part of $\mathcal Y$ is smooth.
Indeed, $\partial_x=y^4+y$ and $\partial_y=x^4+x$ vanish simultaneously
only where $(x^4+x)(y^4+y)=0\neq1$. At $X_\infty=(1:0:0)$ the tangent
cone is $\prod_{\nu\in\F_4}(Y+\nu Z)$, so $X_\infty$ is an ordinary
$4$-fold point. Its four branches correspond to the tangents $Y=\nu Z$.
The same holds at $Y_\infty$ with the tangents $X=\nu Z$. Hence
$g(\mathcal Y)=21-2\cdot6=9$.

\emph{The action on points.}
\begin{itemize}
\item $\varphi_{\alpha,\beta}$ moves every affine point unless
$\alpha=\beta=0$.
\item At $X_\infty$ it sends the branch with tangent $Y=\nu Z$ to the
branch with tangent $Y=(\nu+\beta)Z$.
\item At $Y_\infty$ it sends the branch with tangent $X=\nu Z$ to the
branch with tangent $X=(\nu+\alpha)Z$.
\item $\rho$ interchanges $X_\infty$ and $Y_\infty$.
\end{itemize}

\emph{$u$ acts freely.} In particular $u=\varphi_{1,1}$ has no fixed
point, and Riemann--Hurwitz gives $g(\bar{\mathcal Y})=5$.

\emph{Orbits of $S$ and of $\bar S$.} No nontrivial translation
$\varphi_{\alpha,\beta}$ fixes an affine point.  If two distinct elements
of the coset $S\setminus\{\varphi_{\alpha,\beta}:\alpha,\beta\in\F_4\}$
fixed the same affine point, their quotient would be a nontrivial
translation fixing that point.  Thus an affine stabilizer has order at
most $2$, and every $S$-orbit on affine points has at least $16$ elements.
The eight branches at infinity form a single $S$-orbit.  Since $u$ acts
freely, their images form a $\bar S$-orbit of size $4$, while every
other $\bar S$-orbit has size at least $8$.  In particular, every
$\bar S$-orbit has at least $4$ points.

\emph{Hypotheses (I) and (II).} From the orbit count, $\bar S$ satisfies
\textup{(II)}. It also satisfies \textup{(I)}, since
$|\bar S|=16\ge8$ and $16>2(5-1)$.

\emph{A non-inductive central involution.} Let $\psi$ be the image of
$\varphi_{0,1}$ in $\bar S$, which equals the image of
$\varphi_{1,0}=\varphi_{0,1}u$. It commutes with the image of
$\{\varphi_{\alpha,\beta}\}$. Since $\rho\varphi_{0,1}\rho=\varphi_{1,0}$,
the element $\psi$ is a central involution of $\bar S$.

Its quotient is $\bar{\mathcal Y}/\langle\psi\rangle=\mathcal Y/V$ with
$V=\{\varphi_{\alpha,\beta}:\alpha,\beta\in\F_2\}$, a group of order $4$.
The functions $s=x^2+x$ and $t=y^2+y$ are $V$-invariant and satisfy
$(s^2+s)(t^2+t)=1$. Moreover $[k(x,y):k(s,t)]\le4=|V|$, so
$k(s,t)=k(x,y)^V$.

The curve $(s^2+s)(t^2+t)=1$ of bidegree $(2,2)$ in
$\mathbf P^1\times\mathbf P^1$ is smooth. In the affine chart, the
partial derivatives are $t^2+t$ and $s^2+s$. In the chart $s=1/\sigma$
the partial derivative in $t$ at $\sigma=0$ equals $1$. The point
$(\infty,\infty)$ is not on the curve. Hence the curve has genus $1$,
and $\psi$ is not inductive.

\emph{Determining the case.} $\bar S$ contains the elementary abelian
group $\{\varphi_{\alpha,\beta}\}/\langle u\rangle\cong C_2^3$. A
dihedral group of order $16$ has no such subgroup, so $\bar S$ is not
dihedral. Also $|\bar S|=16\neq2g+2=12$. Because $\psi$ is not
inductive, case \textup{(iii)} of Theorem~\ref{thm:GKclassification} is
excluded. Case \textup{(ii)} is excluded by the order, and case
\textup{(ia)} by non-dihedrality. Hence $(\bar{\mathcal Y},\bar S)$ is
of type \textup{(ib)}.

\emph{Conclusion.} Theorem~\ref{thm:converse} with $h=2$ gives
$\bar{\mathcal Y}\cong X(E,\eta,1,4)$. The elliptic curve $E$ is the
quotient by the non-inductive central involution, namely $\psi$. Hence
$E$ is the curve $(s^2+s)(t^2+t)=1$.
\end{proof}

\begin{remark}\label{rem:GK-caveats}
We record four points about the published version of \cite{GK2015}.
\begin{enumerate}[label=\textup{(\arabic*)}]
\item The displayed equation in Section~6.2 lacks a ``$+$'' between
$YX^7$ and $\mu^5YX^5$.  Equation~\eqref{eq:GK62} is the evident reading.

\item Section~6.2 specifies only that $\mu$ is a primitive element of
$\F_{16}$.  In Proposition~\ref{prop:sporadic-GK}(a) we make the field
presentation explicit by imposing $\mu^4+\mu+1=0$.  The identification
proved there is for this convention; no assertion about a different
choice of primitive element is needed for the argument.

\item At the end of Section~6.2 it is stated that
$\bar{\mathcal X}=\mathcal X/C_2$ has genus $5$, has a dihedral subgroup
of order $8$, and is therefore of type \textup{(ib)}.  That inference
cannot be made from the displayed subgroup: a group of order $8$ on a
genus-$5$ curve does not satisfy the strict inequality in hypothesis
\textup{(I)}, and an acting group in case \textup{(ib)} would have order
$4(g-1)=16$.  The preceding MAGMA computation instead gives the genus-$9$
curve
$\mathcal X$ a subgroup $D_8\times C_2$ of order $32$, which is the
case-\textup{(ib)} curve identified in
Proposition~\ref{prop:sporadic-GK}(a).

\item The statement in Section~6.4 that no nontrivial element of $S$
fixes a point of $\mathcal Y$ is not correct.  The involution
$\rho:(x,y)\mapsto(y,x)$ fixes the four affine points $(a,a)$ with
$a^4+a+1=0$, and $\varphi_{\alpha,0}$ with $\alpha\ne0$ fixes the four
branches at $X_\infty$.  This does not affect the use made here of that
example.  In particular, \cite[Lemma~3.2]{GK2015} states that under
\textup{(I)} and \textup{(II)} the action has two short orbits; what is
required is that $S$ have no \emph{global} fixed point, not that every
nontrivial element act fixed-point-freely.  The orbit computation in the
proof of Proposition~\ref{prop:sporadic-GK}(b) verifies hypothesis
\textup{(II)} directly for both $S$ and $\bar S$.
\end{enumerate}
\end{remark}

\subsection{The parameter space}

\begin{theorem}\label{thm:parameter-space}
Fix $q\ge8$.  For
\[
        E_A:\qquad y^2+xy=x^3+A,\qquad A\in k^\times,
\]
the assignment
\[
        (A,c)\longmapsto X(E_A,c\,dx/x,1,q)
\]
induces a bijection
\[
        (k^\times)^2
        \;\xrightarrow{\ \sim\ }\;
        \{\text{isomorphism classes of type-\textup{(ib)} curves of genus }q+1\}.
\]
\end{theorem}

\begin{proof}
By Proposition~\ref{prop:parameters} and Theorem~\ref{thm:converse}, no
additional parameter occurs.  We recall briefly the normal form for an
ordinary elliptic curve in characteristic $2$.  Since $j(E)\ne0$, a
Weierstrass equation may be normalized to
\[
        y^2+xy=x^3+a_2x^2+a_6,\qquad a_6\ne0.
\]
Replacing $y$ by $y+rx$, where $r^2+r=a_2$ (possible because $k$ is
algebraically closed), removes the $x^2$-term.  Thus
\[
        E\cong E_A:\quad y^2+xy=x^3+A,
        \qquad A=a_6\in k^\times,
\]
and for this model $j(E_A)=A^{-1}$.  Hence $A$ is determined by the
$j$-invariant.  Since $H^0(E_A,\Omega^1_{E_A})$ is one-dimensional and
$dx/x$ is a nonzero invariant differential, every pair $(E,\eta)$ is
represented by
\[
        (E_A,c\,dx/x),\qquad A,c\in k^\times.
\]
For fixed $A$, Theorem~III.10.1 of \cite{Silverman} gives
$\Aut(E_A,0)=\{1,[-1]\}$, and in characteristic $2$ both automorphisms
fix $dx/x$.  Hence two pairs $(E_A,c\,dx/x)$ and
$(E_B,d\,dx/x)$ are isomorphic if and only if $A=B$ and $c=d$.
The assertion follows.
\end{proof}

\begin{remark}
For $q=4$, Theorem~\ref{thm:converse} still gives surjectivity of the
same construction onto the type-\textup{(ib)} curves of genus $5$.
However, Proposition~\ref{prop:parameters} is stated only for $q\ge8$,
because its proof uses uniqueness of the bielliptic involution.  We
therefore do not claim that the parameters $(A,c)$ are unique in genus
$5$.
\end{remark}

\subsection{The unramified cyclic tower}

\begin{proposition}\label{prop:tower}
Let $C=\langle\widetilde\tau\rangle\cong C_q$.  For
$0\le r\le h$, let $C_r\le C$ be the unique subgroup of order $2^r$ and
put
\[
 X_r=X/C_r.
\]
Let $K_r\le T=E[2^h](k)$ be the image of $C_r$ under the projection
$C\to T$, and put
\[
 E_r=E/K_r,\qquad q_r=2^{h-r}.
\]
Then:
\begin{enumerate}[label=\textup{(\roman*)}]
\item $C_r$ acts freely on $X$;
\item $X\to X_r$ is an unramified cyclic cover of degree $2^r$;
\item $X_r$ is ordinary and
\[
 g(X_r)=1+q_r=1+2^{h-r};
\]
\item in particular, $X/C=X_h$ is an ordinary curve of genus $2$;
\item for $0\le r\le h-1$, the quotient by the image of the deck
involution is
\[
 X_r/\langle\bar u\rangle\cong E_r,
\]
and this Artin--Schreier double cover is ramified exactly over
\[
 E_r[2^{h-r}](k),
\]
with lower jump $1$ at every branch point;
\item for $0\le r\le h-1$, choose the unique invariant differential
$\omega_r$ on $E_r$ satisfying $\phi_r^*\omega_r=\omega$, where
$\phi_r:E\to E_r$ is the quotient isogeny.  Then $X_r$ is obtained by the
same residue construction with data
\[
 (E_r,\omega_r,\lambda,q_r).
\]
In particular, if $0\le r\le h-3$, then
\[
 \Aut(X_r)\cong \Dih(C_{q_r})\times C_2,
 \qquad |\Aut(X_r)|=4q_r=4(g(X_r)-1),
\]
and $X_r$ is again of type \textup{(ib)}.
\end{enumerate}
Thus the unramified quotient tower remains inside the same non-dihedral
extremal family from genus $2^h+1$ down to genus $9$.
\end{proposition}

\begin{proof}
Every nontrivial element of $C$ projects to a nontrivial translation of
$E$, hence has no fixed point on $X$.  Thus every $C_r$ acts freely, and
$X\to X_r$ is unramified.  Riemann--Hurwitz yields
\[
 2g(X)-2=2^r\bigl(2g(X_r)-2\bigr),
\]
so
\[
 g(X_r)-1=2^{h-r}=q_r.
\]
For the $2$-rank, the Deuring--Shafarevich formula for the free
$C_r$-action gives
\[
 \gamma(X)-1=2^r\bigl(\gamma(X_r)-1\bigr).
\]
Since $\gamma(X)-1=2^h$, we obtain
\[
 \gamma(X_r)=1+2^{h-r}=g(X_r),
\]
and hence every $X_r$ is ordinary.

Because $C_r$ commutes with $u$, the involution $u$ descends to $X_r$ and
\[
 X_r/\langle\bar u\rangle
 \cong X/\langle C_r,u\rangle
 \cong E/K_r=E_r.
\]
The isogeny $\phi_r:E\to E_r$ is the quotient by a free translation
subgroup, hence is finite \emph{\'etale}.  Moreover,
$C_r\cap\langle u\rangle=1$, so $X$ is the normalization of the fibre
product
\[
 X_r\times_{E_r}E.
\]
Consequently the completed local extension for $X_r\to E_r$ becomes the
completed local extension for $X\to E$ after the unramified base change
$E\to E_r$.  Ramification groups and lower jumps are therefore unchanged.
The branch locus is the image of $T$ in $E_r$, a cyclic group of order
$q_r$.  The elliptic curve $E_r$ is ordinary, being isogenous to $E$, and
therefore this image is precisely $E_r[2^{h-r}](k)$.  Thus every branch
point has lower jump $1$.

Assume now $r\le h-1$, so $q_r$ is even.  Since $\phi_r$ is separable,
pullback on invariant differentials is an isomorphism; hence there is a
unique nonzero invariant differential $\omega_r$ on $E_r$ such that
\[
        \phi_r^*\omega_r=\omega.
\]
The cyclic group $C/C_r\cong C_{q_r}$ acts on $X_r$, commutes with the
image of the deck involution, has trivial intersection with it, and maps
isomorphically onto the translation subgroup $E_r[q_r](k)$.  Together
with the branch and jump computation above, Lemma~\ref{lem:ASdescent} applies and yields a unique $\lambda_r\in k^\times$ such that
\[
        X_r\cong X(E_r,\omega_r,\lambda_r,q_r).
\]

It remains to identify $\lambda_r$.  The curve $X$ is the normalization of
$X_r\times_{E_r}E$, and the map $\phi_r:E\to E_r$ is \emph{\'etale}.
Let $P\in E$ lie above $P_r\in E_r$.  For a local Artin--Schreier
representative $f_r$ of $X_r/E_r$, the map on completed local fields
induced by the finite \emph{\'etale} morphism $\phi_r$ has ramification
index $1$.  Hence pullback preserves residues, and
\[
 \Res_P\bigl(\phi_r^*(f_r\omega_r)\bigr)
   =\Res_{P_r}(f_r\omega_r).
\]  Since
$\phi_r^*\omega_r=\omega$ and the pullback cover is $X/E$, the left-hand
side equals $\lambda$.  Hence
\[
        \lambda_r=\lambda.
\]
Thus
\[
        X_r\cong X(E_r,\omega_r,\lambda,q_r).
\]

If $r\le h-3$, then $q_r\ge8$, so Theorem~\ref{thm:main} applies to this
new set of data.  This gives
\[
 \Aut(X_r)\cong\Dih(C_{q_r})\times C_2
\]
and shows that $X_r$ is again a non-dihedral Nakajima extremal curve of
type \textup{(ib)}.
\end{proof}

\section{Central involutions and their quotients}\label{sec:KR}

The quotients by the central involutions make the inductive structure of
case \textup{(ib)} explicit and connect it with the dihedral case.

\subsection{The centre}

\begin{corollary}\label{cor:center}
For $q\ge8$,
\[
 Z(\Aut(X))=\langle\widetilde\tau^{q/2},u\rangle\cong C_2\times C_2.
\]
Among the three nontrivial central involutions, $u$ is the unique one
whose quotient has genus $1$.  The other two act freely, and each has an
ordinary quotient of genus $q/2+1$.
\end{corollary}

\begin{proof}
The center of $\Dih(C_q)$ is generated by the unique element of order
$2$ in $C_q$, namely $\widetilde\tau^{q/2}$.  The description of the
center follows from Proposition~\ref{prop:fullaut}.  The involution $u$ has quotient
$E$ of genus one.  Both $\widetilde\tau^{q/2}$ and
$u\widetilde\tau^{q/2}$ project to the nonzero $2$-torsion translation
on $E$, so they have no fixed point.  Riemann--Hurwitz and
Deuring--Shafarevich for either free involution give genus and $2$-rank
$q/2+1$ for the quotient.
\end{proof}

\begin{theorem}\label{thm:centralquotients}
Assume $q=2^h$ with $h\ge4$, put
\[
 t=\widetilde\tau,\qquad a=t^{q/2},\qquad
 V=\langle a,u\rangle=Z(\Aut(X)),
\]
and set
\[
 X^-:=X/\langle a\rangle,\qquad
 X^+:=X/\langle au\rangle,\qquad
 E_1:=X/V.
\]
Then:
\begin{enumerate}[label=\textup{(\roman*)}]
\item $E_1$ is an ordinary elliptic curve and the natural map
      $E=X/\langle u\rangle\to E_1$ is a separable isogeny of degree $2$;
\item $X^-$ is an ordinary Nakajima extremal curve of genus
\[
 g(X^-)=\frac q2+1
\]
whose full automorphism group is
\[
 \Aut(X^-)\cong \Dih(C_{q/2})\times C_2.
\]
Moreover, for any nonzero invariant differential $\omega_1$ on $E_1$,
there is a $\lambda_1\in k^\times$ such that
\[
 X^-\cong X(E_1,\omega_1,\lambda_1,q/2).
\]
In particular $X^-$ is again a member of the constructed family and is of
type \textup{(ib)};
\item $X^+$ is an ordinary Nakajima extremal curve of the same genus, but
\[
 \Aut(X^+)\cong \Dih(C_q);
\]
in particular $X^+$ is of the dihedral type \textup{(ia)};
\item relative to the pair $(X,G)$, the involutions $a$ and $au$ are
      exactly the two central inductive involutions in the sense of
      \cite{GK2015}; the bielliptic involution $u$ is not inductive.
\end{enumerate}
\end{theorem}

\begin{proof}
By Corollary~\ref{cor:center}, both $a$ and $au$ act freely on $X$.  Hence
Riemann--Hurwitz and Deuring--Shafarevich give
\[
 g(X^-)=g(X^+)=\frac q2+1,
 \qquad
 \gamma(X^-)=\gamma(X^+)=\frac q2+1.
\]
Thus both quotients are ordinary.  Moreover,
\[
 E_1=X/\langle a,u\rangle
     \cong E/\langle \bar a\rangle,
\]
where $\bar a$ is translation by the unique nonzero $2$-torsion point of
$E$.  The map $E\to E_1$ is therefore the separable degree-$2$ isogeny
with kernel $\langle\bar a\rangle$, and $E_1$ is ordinary.  Put
\[
        T_1:=E_1[2^{h-1}](k).
\]

Consider first $X^-$.  The group $G/\langle a\rangle$ acts faithfully on
$X^-$ and
\[
 G/\langle a\rangle
 \cong \Dih(C_{q/2})\times C_2,
\]
of order $2q=4(g(X^-)-1)$.  The image of $u$ is central and
\[
 X^-/\langle\bar u\rangle=E_1.
\]
The subgroup $\langle a\rangle$ is the unique subgroup of order $2$ in
$C=\langle t\rangle$.  Thus $X^-=X/\langle a\rangle$ is precisely the
case $r=1$ of Proposition~\ref{prop:tower}, and $E_1=E/\langle\bar a\rangle$
is the corresponding elliptic quotient.  Let $\omega_1^0$ be the unique
invariant differential on $E_1$ satisfying
\[
        (E\to E_1)^*\omega_1^0=\omega.
\]
Proposition~\ref{prop:tower}(vi) gives directly
\[
        X^-\cong X(E_1,\omega_1^0,\lambda,q/2).
\]
If $\omega_1$ is any other nonzero invariant differential on $E_1$, write
$\omega_1=c\omega_1^0$ with $c\in k^\times$.  The defining residue
condition then gives
\[
        X(E_1,\omega_1^0,\lambda,q/2)
        \cong X(E_1,\omega_1,c\lambda,q/2).
\]
Thus the assertion in part \textup{(ii)} holds with $\lambda_1=c\lambda$.

Since $h\ge4$, one has $q/2=2^{h-1}\ge8$.  The displayed
identification realizes $X^-$ as a curve of Theorem~\ref{thm:main} with
parameter $q/2$.  Thus Theorem~\ref{thm:main}(iii)--(iv) gives directly
\[
 \Aut(X^-)\cong\Dih(C_{q/2})\times C_2,
\]
and shows that $X^-$ is of type \textup{(ib)}.

Now put $H=\langle au\rangle$.  In $G/H$ the images of $a$ and $u$
coincide, while the image of $t$ still has order $q$.  Hence
\[
 G/H\cong\Dih(C_q),
\]
of order $2q=4(g(X^+)-1)$.  The common image $\bar a=\bar u$ is a
bielliptic involution and
\[
 X^+/\langle\bar u\rangle=E_1.
\]
We claim that the branch locus is again exactly $T_1$.  A point
$\overline P\in X^+=X/H$ is fixed by the common image
$\bar u=\bar a$ precisely when, for a lift $P\in X$, one has either
$uP=P$ or $aP=P$: indeed $uP$ must lie in the $H$-orbit of $P$, and
$H=\{1,au\}$.  The second alternative is impossible because $a$ acts
freely.  Thus the fixed points of $\bar u$ are precisely the images of the
$u$-fixed points on $X$.  Their images in $E_1$ are the image of $T$ under
the degree-$2$ isogeny $E\to E_1$, which is precisely $T_1$.
Hence this is exactly the branch locus of $X^+\to E_1$.

Since $h\ge4$, one has $g(X^+)=q/2+1\ge9$, so the bielliptic involution
of $X^+$ is unique by Castelnuovo--Severi.  By uniqueness of the
bielliptic involution, every automorphism descends to
$(E_1,T_1)$.  Here $|T_1|=q/2$, and, as in
Proposition~\ref{prop:fullaut},
\[
        \Aut(E_1,T_1)=T_1\rtimes\langle[-1]\rangle
\]
has order $q$.  Thus $|\Aut(X^+)|\le2q$, and the dihedral subgroup
$G/H$ already has this order.  Therefore
\[
 \Aut(X^+)\cong\Dih(C_q).
\]
The image of $t$ on $E_1$ is a nontrivial translation, so neither of the
two displayed extremal groups has a global fixed point.

For $X^+$ the acting group is dihedral of order
$4(g(X^+)-1)=2q$.  Case \textup{(ii)} would require
$2q=2g(X^+)+2=q+4$, hence $q=4$, whereas here $h\ge4$.  Moreover
$Z(\Dih(C_q))=\langle\bar u\rangle$, and
$X^+/\langle\bar u\rangle=E_1$ has genus $1$.  Thus $\bar u$ is
non-inductive, so case \textup{(iii)} is also excluded.  The
classification therefore places $X^+$ in case \textup{(ia)}.
Moreover,
\[
 |G/\langle a\rangle|
 =|G/\langle au\rangle|
 =2q
 =4\left(\frac q2+1-1\right),
\]
and both induced groups have order at least $8$ and no global fixed point.
Thus $a$ and $au$ are inductive central involutions.  On the other hand,
$X/\langle u\rangle=E$ has genus one, so $u$ is not inductive.  Since
$Z(G)=\langle a,u\rangle$ has exactly three nontrivial elements, $a$ and
$au$ are precisely the two central inductive involutions.

\end{proof}

\begin{remark}\label{rem:central-h3}
When $h=3$ (so $q=8$), the quotient $X^-=X/\langle a\rangle$ has genus
$5$.  Proposition~\ref{prop:tower}(vi) identifies it with a curve of the
form $X(E_1,\omega_1,\lambda_1,4)$, and Remark~\ref{rem:smallq} shows
that it is still of type \textup{(ib)}.  In this genus we do not claim a
determination of the full automorphism group.
\end{remark}

\section{Explicit equations and a logarithmic-differential realization}\label{sec:explicit}

The construction admits the following explicit form.  The proposition is stated for residue $1$; multiplication by $\lambda\in k^\times$ gives the general case.

\begin{proposition}\label{prop:logdiff}
Let $q=2^h$ and let $E/k$ be an ordinary elliptic curve.  Choose
$R\in E(k)$ of order $2q$ and put
\[
 T=\langle 2R\rangle.
\]
Then there is $H\in k(E)^\times$ with
\[
 \operatorname{div}(H)=\sum_{P\in T}[P]-q[R].
\]
If $\omega$ is a nonzero invariant differential, then
\[
 f=\frac{dH/H}{\omega}
\]
has a simple pole at every point of $T$ and no other pole, and
\[
 \Res_P(f\omega)=1\qquad(P\in T).
\]
\end{proposition}

\begin{proof}
The subgroup $T$ is cyclic of even order $q$, so the sum of its elements
in the group law of $E$ is its unique nonzero point of order $2$, namely
$qR$.  Therefore the degree-zero divisor
\[
 \sum_{P\in T}[P]-q[R]
\]
has trivial class in $\Pic^0(E)$ and is principal.

At a zero $P\in T$ of $H$, the logarithmic differential $dH/H$ has a
simple pole with residue $1$.  At $R$, write locally $H=t^{-q}v$ with
$v$ a unit.  Then
\[
 \frac{dH}{H}=-q\frac{dt}{t}+\frac{dv}{v}=\frac{dv}{v},
\]
because $q=0$ in characteristic $2$.  Hence $dH/H$ is regular at $R$.
Dividing by the nowhere-vanishing differential $\omega$ proves the
claim.
\end{proof}

\begin{proposition}[A genus-$9$ equation over $\F_2$]\label{prop:explicit9}
Let $k=\overline{\F}_2$ and consider the ordinary elliptic curve
\[
 E:\qquad y^2+xy=x^3+1.
\]
Let $\rho\in\F_4\setminus\F_2$ satisfy $\rho^2+\rho+1=0$.  Then
$P_0=(\rho,0)$ has order $8$, and
\[
 T=E[8](k)=E(\F_4)
 =\{O,(0,1)\}\cup
   \{(a,0),(a,a):a\in\F_4^\times\}.
\]
With the invariant differential $\omega=dx/x$, put
\begin{equation}\label{eq:explicit-f9}
 f=\frac{y}{x}+\frac{x^4+x^3+1}{x^4+x}.
\end{equation}
Then $f$ has simple poles precisely at the points of $T$ and
\[
 \Res_P(f\omega)=1\qquad(P\in T).
\]
Consequently the smooth projective model of
\begin{equation}\label{eq:explicit-AS9}
 z^2+z=\frac{y}{x}+\frac{x^4+x^3+1}{x^4+x},
 \qquad y^2+xy=x^3+1,
\end{equation}
is an ordinary curve of genus $9$ and
\[
 \Aut(X)\cong \Dih(C_8)\times C_2.
\]
It is birational to the plane curve
\begin{equation}\label{eq:explicit-plane9}
 \boxed{(x^7+x)(z^4+z)=x^8+x^6+x^4+x^2+1}.
\end{equation}
In particular, this genus-$9$ member of the family has a plane model over
$\F_2$.
\end{proposition}

\begin{proof}
The curve $E$ is ordinary, and a direct use of the chord--tangent law gives
\[
 [2](\rho,0)=(1,0),\qquad [4](\rho,0)=(0,1).
\]
Since $(0,1)$ is the nontrivial $2$-torsion point, $P_0$ has order $8$.
The displayed set has eight points and is exactly $\langle P_0\rangle$.

We verify the principal parts of $f$.  First let $P=(a,b)$ with
$a\in\F_4^\times$.  Then $a^3=1$, the coordinate $x-a$ is a local
parameter at $P$, and the second term of \eqref{eq:explicit-f9} is the
only singular term.  Since
\[
 x^4+x=x(x^3+1),\qquad
 \frac{d}{dx}(x^4+x)=1,
\]
and $a^4+a^3+1=a$, one obtains
\[
 \Res_P(f\omega)=\frac{a}{a}=1.
\]
This accounts for the six points of $T$ with nonzero $x$-coordinate.

At $P=(0,1)$ put $s=y+1$.  The equation of $E$ becomes
\[
 s^2=x(1+s+x^2),
\]
so $s$ is a local parameter and
\[
 \omega=\frac{dx}{x}=\frac{ds}{1+s+x^2}.
\]
Writing \eqref{eq:explicit-f9} over the common denominator
$x(1+x^3)$ gives
\[
 f=\frac{s(1+x^3)+x^4}{x(1+x^3)}.
\]
Therefore
\[
 f\omega
 =\left(\frac1s+
   \frac{x^4}{s^2(1+x^3)}\right)ds.
\]
As $\operatorname{ord}_s(x)=2$, the second summand is regular, and the
residue is again $1$.  Finally, at $O$ one has
$\operatorname{ord}_O(x)=-2$ and $\operatorname{ord}_O(y)=-3$, so $f$
has a simple pole there and no further poles occur.  The residue theorem,
together with the seven finite residues already computed, gives
$\Res_O(f\omega)=1$.

Thus \eqref{eq:explicit-AS9} is the case $q=8$, $\lambda=1$ of the main
construction, and Theorem~\ref{thm:main} yields the genus, ordinarity, and
full automorphism group.

For the plane equation set $W=z^2+z$.  From
\eqref{eq:explicit-AS9},
\[
 y=xW+\frac{x^4+x^3+1}{x^3+1}.
\]
Substitution in $y^2+xy=x^3+1$ and simplification in characteristic $2$
gives
\[
 W^2+W
 =\frac{x^8+x^6+x^4+x^2+1}{x^7+x}.
\]
Since $W^2+W=z^4+z$, this is precisely
\eqref{eq:explicit-plane9}.  Conversely, the displayed expression for $y$ shows that the function
field defined by \eqref{eq:explicit-plane9} contains the function field
of \eqref{eq:explicit-AS9}.  The latter has degree $4$ over $k(x)$,
while \eqref{eq:explicit-plane9} has degree $4$ in $z$.  Hence the
polynomial in $z$ in \eqref{eq:explicit-plane9} is irreducible over
$k(x)$ and the two function fields coincide.  Thus the two models are
birational.
\end{proof}

\begin{remark}[The case $q=4$]\label{rem:smallq}
The residue construction also works for $q=4$ and gives an ordinary
curve of genus $5$ with a subgroup
\[
        S\cong\Dih(C_4)\times C_2
\]
of order $16=4(g-1)$.  The same construction shows that $S$ has no
global fixed point, while the central deck involution $u$ has elliptic
quotient and hence is not inductive.  Moreover $S$ is not dihedral,
because its center has order $4$ whereas the center of a dihedral group
of order $16$ has order $2$, and $16\ne2g+2=12$.  Thus cases
\textup{(ia)}, \textup{(ii)}, and \textup{(iii)} of
Theorem~\ref{thm:GKclassification} are excluded, so the constructed
curve is of type \textup{(ib)}.

Theorem~\ref{thm:converse} shows conversely that every curve of type
\textup{(ib)} of genus $5$ arises in this way.  Proposition~\ref{prop:sporadic-GK}
provides a direct verification for the genus-$5$ example arising from
\cite[Sections~6.4--6.5]{GK2015}.  The uniqueness argument for the
bielliptic involution does not apply in genus $5$, so we do not claim
that the parameters $(A,c)$ are unique.  Nakajima's bound shows only
that the displayed group is a Sylow $2$-subgroup of the full
automorphism group.  We therefore do not extend the assertion on the
full automorphism group to $q=4$.
\end{remark}

\section*{Acknowledgments}

The author was partially supported by CNPq (grant no.~302774/2025-4),
FAEPEX (grant no.~3485/25), and FAPESP (grant no.~2024/00923-6).

\section*{Declaration on the use of generative AI}

During the preparation of this manuscript, the author used ChatGPT (OpenAI) as an auxiliary tool for language editing, improving exposition, and checking intermediate mathematical arguments. All statements, proofs, computations, citations, and conclusions were subsequently reviewed and verified by the author, who takes full responsibility for the content of the paper.

\bigskip
\noindent
\textit{Departamento de Matem\'atica, IMECC, Universidade Estadual de Campinas,\\
Rua S\'ergio Buarque de Holanda 651, Campinas, SP 13083-859, Brazil}\\
\texttt{saeed@unicamp.br}

\end{document}